\documentclass[12pt, twoside, leqno]{article}

\usepackage{amsmath,amsthm}
\usepackage{amssymb}

\usepackage{enumitem}

\usepackage{graphicx}
\usepackage{url}

\usepackage[T1]{fontenc}

\newtheorem{theorem}{Theorem}[section]
\newtheorem{corollary}[theorem]{Corollary}

\newtheorem{proposition}[theorem]{Proposition}

\theoremstyle{definition}

\newtheorem{remark}[theorem]{Remark}

\numberwithin{equation}{section}

\allowdisplaybreaks[4]

\begin{document}


\baselineskip=17pt


\title{Multiple  Mertens evaluations}

\author{
Tianfang Qi\\
Department of Mathematics\\
Nanjing University\\
Nanjing  210093, China\\
E-mail: 15914449424@163.com
\and
Su Hu\\
Department of Mathematics\\
South China University of Technology\\
Guangzhou 510640, China\\
E-mail: mahusu@scut.edu.cn}

\date{}

\maketitle


\renewcommand{\thefootnote}{}



\renewcommand{\thefootnote}{\arabic{footnote}}
\setcounter{footnote}{0}


\begin{abstract}
The Mertens' first theorem gives us the following asymptotic formula
\begin{equation*}
  \sum_{\substack{p\leq x\\ p~prime}}\frac{\log p}{p}=\log x+O(1),
\end{equation*}
and the Mertens' second theorem indicates that there exists a constant $B\approx 0.261$, named the Mertens constant, such that
\begin{equation*}
 \sum_{\substack{p\leq x\\ p~prime}}\frac{1}{p}=\log_{2}x+B+O\left(\frac{1}{\log x}\right).
\end{equation*}
In this paper, by using the Abel summation formula and Dirichlet's hyperbola method, we extend them to multiple cases.
\end{abstract}

\section{Introduction}
Throughout this paper we need the following notation. Denote $\log_{2}x=\log(\log x)$ as the iterated natural logarithm, and $p$ as a prime number. For a fixed number $a\in\mathbb{R}\cup\{-\infty\}$ and real-valued functions $g:(a,\infty)\to[0,\infty)$, $f:(a,\infty)\to \mathbb{R},$ $f(x)=O(g(x))$ or $f(x)\ll g(x)$ means there exists $M>0,~b\geq a$, such that $|f(x)|\leq Mg(x)$ for any $x\geq b$.

 In 1874, Mertens \cite{Mertens} proved the following two interesting and beautiful theorem (\cite[p. 89--90, Theorem 4.10 and 4.12]{Apostol}).

 \noindent Mertens' first theorem:
\begin{equation}\label{first}
  \sum_{p\leq x}\frac{\log p}{p}=\log x+O(1);
\end{equation}

\noindent Mertens' second theorem:
\begin{equation}\label{Mertens}
  \sum_{p\leq x}\frac{1}{p}=\log_{2}x+B+O\left(\frac{1}{\log x}\right),
\end{equation}
where $B\approx 0.261$ is the Mertens constant.

The above Mertens'  theorems have many applications in modern number theory and have appeared in many works (for example, \cite[Theorem 4.12]{Apostol} and \cite[p. 13]{Everest}). One  direction  on investigating Mertens'  second theorem (\ref{Mertens}) is to increase the number of prime variables.
In 2002, Saidak \cite{Sadak} presented the following double Mertens type evaluation
\begin{equation}\label{sadak}
  \sum_{pq\leq x}\frac{1}{pq}=(\log_{2}x+B)^{2}-\frac{\pi^{2}}{6}+O\left(\frac{\log_{2}x}{\log x}\right)
\end{equation}
and in 2014, Popa \cite{Popa2} also obtained the following evaluation
\begin{equation}\label{popa2}
  \sum_{pq\leq x}\frac{1}{pq}=(\log_{2}x+B)^{2}-\log^{2}2+2\int_{0+0}^{\frac{1}{2}}\frac{\log(1-x)}{x}dx+O\left(\frac{\log_{2}x}{\log x}\right).
\end{equation}
From the equality $-\log^{2}2+2\int_{0+0}^{\frac{1}{2}}\frac{\log(1-x)}{x}dx=-\frac{\pi^{2}}{6}$~(\cite[p. 5, (1.11)]{Lewin}), we see that (\ref{popa2}) implies (\ref{sadak}).

In 2016, Popa \cite{Popa3} further proved the following triple Mertens evaluation
\begin{equation}\label{popa3}
  \sum_{pqr\leq x}\frac{1}{pqr}=(\log_{2}x+B)^{3}-\frac{\pi^{2}}{2}(\log_{2}x+B)+2\zeta(3)+O\left(\frac{(\log_{2}x)^{2}}{\log x}\right),
\end{equation}
where $$\zeta(s)=\sum_{n=1}^{\infty}\frac{1}{n^{s}},~~\rm{Re}(s) >1$$  is the Riemann zeta function.

In this paper, by using the Abel summation formula and Dirichlet's hyperbola method, we extend the above Mertens'  formulas (\ref{first}) and (\ref{Mertens}) to multiple cases.
Our main results are as follows.
\begin{theorem}\label{lem.B} Set  $P_{0}(y)=1$ and $P_{1}(y)=y+B$. For any $k\geq 2$, let $P_{k}(x)$ be given in (\ref{simplified}). Then for any positive integers $k$ and $s$, the following evaluation holds
\begin{equation}\label{main1}
\begin{aligned}
  \sum_{p_{1}\cdots{p_{k}}\leq{x}}\frac{\log^{s}(p_{1}\cdots p_{k})}{p_{1}\cdots p_{k}}
    =&\sum_{l=0}^{k-1}(-1)^{l}\frac{A_{k}^{l+1}}{s^{l+1}}P_{k-1-l}(\log_{2}x)\cdot \log^{s}x
     +f(2)\log^{s-1}2\\
     &+O(\log^{s-1}x\cdot (\log_{2}x)^{k}),
     \end{aligned}
\end{equation}
where
\begin{align*}
  f(x)=\sum_{l=0}^{k-1}(-1)^{l}A_{k}^{l+1}P_{k-1-l}(\log_{2}x)\cdot \log x.
\end{align*}
and the combinatorial number $A_{k}^{l}=\binom{k}{l}\cdot l!$.
\end{theorem}

In the case $s=1$, $k=1$, Theorem \ref{lem.B} reduces to Mertens' first theorem (\ref{first}), and $s=1$, $k=2$ is Theorem 3.3 in  B\u{a}nescu-Popa \cite{B\u{a}nescu}. Moreover, the following results follow immediately.
\begin{corollary}\label{cor1}
We have the following evaluation
\begin{equation*}
    \frac{1}{(\log x)^{s}}\sum_{p_{1}\cdots{p_{k}}\leq{x}}\frac{\log^{s}(p_{1}\cdots p_{k})}{p_{1}\cdots p_{k}}
       =\sum_{l=0}^{k-1}(-1)^{l}\frac{A_{k}^{l+1}}{s^{l+1}}P_{k-1-l}(\log_{2}x)+O\left(\frac{(\log_{2}x)^{k}}{\log x}\right).
\end{equation*}
\end{corollary}

\begin{corollary}\label{cor2}
We have the following evaluation
\begin{align*}
    \frac{1}{(\log x)^{s}}\sum_{p_{1}\cdots{p_{k}}\leq{\sqrt{x}}}\frac{\log^{s}(p_{1}\cdots p_{k})}{p_{1}\cdots p_{k}}
      =\frac{1}{2^{s}}\sum_{l=0}^{k-1}(-1)^{l}\frac{A_{k}^{l+1}}{s^{l+1}}P_{k-1-l}(\log_{2}\sqrt{x})
         +O\left(\frac{(\log_{2}x)^{k}}{\log x}\right).
\end{align*}
\end{corollary}

\begin{theorem}\label{thm}
For any positive integer $k$, we have
\begin{equation}\label{Main theorem}
\begin{aligned}
  \sum_{p_{1}\cdots p_{k}\leq x}\frac{1}{p_{1}\cdots p_{k}}&=(\log_{2}x+B)^{k}+\sum_{m=2}^{k}C_{k}^{m}a_{m}(\log_{2}x+B)^{k-m}\\&\quad+O\left(\frac{(\log_{2}x)^{k-1}}{\log x}\right),
\end{aligned}
\end{equation}
where $\{a_{n}\}$ is a sequence related to the Riemann zeta function $\zeta$, that is,
\begin{equation}\label{an}
  \begin{split}
    &a_{2}=-\zeta(2), a_{3}=2\zeta(3), a_{4}=3\zeta(2)^{2}-6\zeta(4),\\
    &a_{k}=\sum_{i=1}^{k-3}(-1)^{i}C_{k-1}^{i}i!\zeta(i+1)a_{k-1-i}+(-1)^{k-1}(k-1)!\zeta(k)  ~~(k>{4}).
  \end{split}
\end{equation}
and $C_{k}^{l}=\binom{k}{l}$.
\end{theorem}

\begin{remark} The cases $k=1,2$ and $3$ are Mertens' second theorem (\ref{Mertens}), \cite[Theorem 1]{Popa2} and  \cite[Theorem 11]{Popa3}, respectively. \end{remark}

\begin{remark}
Let $\Gamma$ be the Euler gamma function and $\gamma$ be the Euler constant. Based on the argument of the Selberg-Delange method and the complex function theory, in 2016, Tenenbaum obtained the following multiple Mertens evaluation for $x\geq 3$ (see \cite[Theorem 1]{Tenen2016})
\begin{equation}\label{Ten16} \sum_{p_{1}\cdots p_{k}\leq x}\frac{1}{p_{1}\cdots p_{k}}=P_{k}(\log_{2}x)+O\left(\frac{(\log_{2}x)^{k}}{\log x}\right),\end{equation}
where $P_{k}(X):=\sum_{0\leq j\leq k}\lambda_{j,k}X^{j}$, and
$$\lambda_{j,k}:=\sum_{0\leq m\leq k-j}\dbinom{k}{m,j,k-m-j}(B-\gamma)^{k-m-j}\left(\frac{1}{\Gamma}\right)^{(m)}(1) ~~(0\leq j\leq k).$$
In a recent version of the above article \cite{Tenen2019}, Tenenbaum showed that his method may provide the same error term as Theorem \ref{thm}.

The main term in (\ref{Ten16}) is expressed by the higher order derivatives of the gamma function, while the main term in our formula (\ref{Main theorem})
is related to the special values of Riemann zeta function. There seems to be no direct connections between  them.
But a recent work \cite{J-P-J} shows that (\ref{Main theorem}) and (\ref{Ten16}) are equivalent (see  \cite[Sec. 6]{J-P-J}).
Our approach is based on the Abel summation formula (see \cite{B\u{a}nescu}) and Dirichlet's hyperbola method, it is elementary and makes no use of complex function theory.
Furthermore,  during our approach we also get a multiple generalization of Dirichlet's hyperbola method (see Proposition \ref{prop.sep}).
\end{remark}
\begin{remark}
Korolev \cite[p. 17--33]{Korolev} also calculated some other types multiple sums with primes.
\end{remark}
\begin{remark}
Recently,  building on (\ref{Main theorem}) and (\ref{Ten16}), Bayless, Kinlaw and Lichtman \cite{J-P-J} gave elementary proofs of precise asymptotics for the reciprocal sum of $k$-almost primes, 
$$\begin{aligned}
{\mathcal R}_k(x) = \mathop{\sum_{\Omega(n)=k}}_{n\le x}\frac{1}{n} = \mathop{\sum_{p_1\cdots p_k\le x}}_{p_1\le \ldots\le p_k}\frac{1}{p_1\cdots p_k},
\end{aligned}
$$
where for a positive integer $n,$ $\Omega(n)$ denotes the number of prime factors of $n$, counted with multiplicity. 
\end{remark}

\section{Preliminaries}
\subsection{The hyperbola method of Dirichlet for a multiple sum}

The main purpose of this subsection is to extent the classical hyperbola method of Dirichlet (see \cite[Theorem 3.17]{Apostol}), especially the triple sum given in \cite{Popa3} to multiple cases.
\begin{proposition}\label{prop.sep}
$(1)$ Let $\psi:\Bbb{N}^{k}\to\Bbb{R}$ be a real-valued function. Then for each $0<y<x$, we have the following identity
\begin{equation}\label{2.1(1)}
  \sum_{i_{1}\cdots{i_{k}}\leq x}\psi(i_{1},\cdots,i_{k})=\sum_{i_{k}\leq{y}}\sum_{i_{1}\cdots{i_{k-1}}\leq\frac{x}{i_{k}}}\psi(i_{1},\cdots,i_{k})
  +\sum_{i_{1}\cdots{i_{k-1}}\leq\frac{x}{y}}\sum_{y<i_{k}\leq\frac{x}{i_{1}\cdots{i_{k-1}}}}\psi(i_{1},\cdots,i_{k}).
\end{equation}
$(2)$ Let $f_{1},\cdots,f_{k}:\Bbb{N}\to\Bbb{R}$ be real-valued functions and define $S_{f_{k}},S_{f_{1},\cdots,f_{k-1}}:(0,\infty)\to\Bbb{R}$ by
\begin{equation*}S_{f_{k}}(x)=\sum_{i_{k}\leq x}f_{k}(i_{k})~~ \textrm{and}~~S_{f_{1},\cdots,f_{k-1}}(x)=\sum_{i_{1}\cdots{i_{k-1}}\leq x}f_{1}(i_{1})\cdots f_{k-1}(i_{k-1}).
\end{equation*}
Then for each $0<y<x$, we have the following identity
\begin{equation}
\begin{aligned}
  \sum_{i_{1}\cdots{i_{k}}\leq x}f_{1}(i_{1})\cdots f_{k}(i_{k})
  =&\sum_{i_{k}\leq{y}}f_{k}(i_{k})S_{f_{1},\cdots,f_{k-1}}\Big(\frac{x}{i_{k}}\Big)\\
  &+\sum_{i_{1}\cdots{i_{k-1}}\leq\frac{x}{y}}f_{1}(i_{1})\cdots f_{k-1}(i_{k-1})S_{f_{k}}\Big(\frac{x}{i_{1}\cdots{i_{k-1}}}\Big)\\
   &-S_{f_{k}}(y)S_{f_{1},\cdots,f_{k-1}}\Big(\frac{x}{y}\Big).
\end{aligned}
\end{equation}
\end{proposition}

\begin{proof} Our method  employed here is analogue to \cite[Proposition 3]{Popa3}.

$(1)$
We have
\begin{align*}
  \sum_{i_{1}\cdots{i_{k}}\leq x}\psi(i_{1},\cdots,i_{k})
  =\sum_{i_{1}\cdots{i_{k}}\leq x, i_{k}\leq y}\psi(i_{1},\cdots,i_{k})
   +\sum_{i_{1}\cdots{i_{k}}\leq x, y<i_{k}}\psi(i_{1},\cdots,i_{k}),
\end{align*}
since
\begin{align*}
  \left\{(i_{1},\cdots, i_{k})\in\mathbb{N}^{k}~|~i_{1}\cdots i_{k}\leq x\right\}
    =&\left\{(i_{1},\cdots, i_{k})\in\mathbb{N}^{k}~|~i_{1}\cdots i_{k}\leq x, ~i_{k}\leq y\right\}\\
     &\cup\left\{(i_{1},\cdots, i_{k})\in\mathbb{N}^{k}~|~i_{1}\cdots i_{k}\leq x, ~y<i_{k}\right\}
\end{align*}
and the sets on the right side of the equation are disjoint.

Clearly $$\sum_{i_{1}\cdots{i_{k}}\leq x, i_{k}\leq y}\psi(i_{1},\cdots,i_{k})=\sum_{i_{k}\leq y}
\sum_{i_{1}\cdots{i_{k-1}}\leq\frac{x}{i_{k}} }\psi(i_{1},\cdots,i_{k}),$$
which follows that
\begin{align}\label{*}
  \sum_{i_{1}\cdots{i_{k}}\leq x}\psi(i_{1},\cdots,i_{k})
  =\sum_{i_{k}\leq y}\sum_{i_{1}\cdots{i_{k-1}}\leq\frac{x}{i_{k}} }\psi(i_{1},\cdots,i_{k})
   +\sum_{i_{1}\cdots{i_{k}}\leq x}\sum_{y<i_{k}}\psi(i_{1},\cdots,i_{k}).
\end{align}
It is easy to check that
\begin{align*}
  &\left\{(i_{1},\cdots, i_{k})\in\mathbb{N}^{k}~|~i_{1}\cdots i_{k}\leq x, ~y<i_{k}\right\}\\
              =&\left\{(i_{1},\cdots, i_{k})\in\mathbb{N}^{k}~\Big|~y<i_{k}\leq\frac{x}{i_{1}\cdots i_{k-1}}, ~i_{1}\cdots i_{k-1}\leq\frac{x}{y}\right\}.
\end{align*}
Thus we obtain
\begin{align}\label{**}
  \sum_{i_{1}\cdots{i_{k}}\leq x}\sum_{y<i_{k}}\psi(i_{1},\cdots,i_{k})
   =\sum_{i_{1}\cdots i_{k-1}\leq\frac{x}{y}}\sum_{y<i_{k}\leq\frac{x}{i_{1}\cdots i_{k-1}}}\psi(i_{1},\cdots,i_{k}).
\end{align}
Now we have completed the proof for the first part by substituting (\ref{**}) into (\ref{*}).

$(2)$ By (\ref{2.1(1)}), we have
\begin{align*}
  \sum_{i_{1}\cdots{i_{k}}\leq x}f_{1}(i_{1})\cdots f_{k}(i_{k})
  =&\sum_{i_{k}\leq{y}}f_{k}(i_{k})\left(\sum_{i_{1}\cdots i_{k-1}\leq\frac{x}{i_{k}}}f_{1}(i_{1})\cdots f_{k-1}(i_{k-1})\right)\\
    &+\sum_{i_{1}\cdots{i_{k-1}}\leq\frac{x}{y}}f_{1}(i_{1})\cdots f_{k-1}(i_{k-1})\left(\sum_{y<i_{k}\leq\frac{x}{i_{1}\cdots i_{k-1}}}f_{k}(i_{k})\right)\\
  =&\sum_{i_{k}\leq{y}}f_{k}(i_{k})S_{f_{1},\cdots,f_{k-1}}\Big(\frac{x}{i_{k}}\Big)\\
   &+\sum_{i_{1}\cdots{i_{k-1}}\leq\frac{x}{y}}f_{1}(i_{1})\cdots f_{k-1}(i_{k-1})\left(S_{f_{k}}\Big(\frac{x}{i_{1}\cdots{i_{k-1}}}\Big)-S_{f_{k}}(y)\right)\\
  =&\sum_{i_{k}\leq{y}}f_{k}(i_{k})S_{f_{1},\cdots,f_{k-1}}\Big(\frac{x}{i_{k}}\Big)\\
   &+\sum_{i_{1}\cdots{i_{k-1}}\leq\frac{x}{y}}f_{1}(i_{1})\cdots f_{k-1}(i_{k-1})
       S_{f_{k}}\Big(\frac{x}{i_{1}\cdots{i_{k-1}}}\Big)\\
   &-S_{f_{k}}(y)S_{f_{1},\cdots,f_{k-1}}\Big(\frac{x}{y}\Big),
\end{align*}
which completes the proof.
\end{proof}

Letting $f_{1}=\cdots=f_{k}=f$ and $y=\sqrt{x}$ in Proposition \ref{prop.sep}, we immediately have the following corollaries.
\begin{corollary}\label{cor}
Let $f:\Bbb{N}\to\Bbb{R}$ be a real-valued function, and define $S_{1}, S_{k-1}:(0,\infty)\to\Bbb{R}$ by \begin{equation*} S_{1}(x)=\sum_{i\leq x}f(i)~~\textrm{and}~~S_{k-1}(x)=\sum_{i_{1}\cdots{i_{k-1}}\leq x}f(i_{1})\cdots f(i_{k-1}),
\end{equation*}
 we have
\begin{equation*}
\begin{split}
  \sum_{i_{1}\cdots{i_{k}}\leq x}f(i_{1})\cdots f(i_{k})=&\sum_{i_{k}\leq{\sqrt{x}}}f(i_{k})S_{k-1}\Big(\frac{x}{i_{k}}\Big)
  +\sum_{i_{1}\cdots{i_{k-1}}\leq{\sqrt{x}}}f(i_{1})\cdots f(i_{k-1})S_{1}\Big(\frac{x}{i_{1}\cdots{i_{k-1}}}\Big)\\
  &-S_{1}(\sqrt{x})S_{k-1}(\sqrt{x}).
\end{split}
\end{equation*}
\end{corollary}

\begin{corollary}\label{cor.sep}
Let $\Bbb{P}$ be the set of all prime numbers and $u:\Bbb{P}\to\Bbb{R}$ be a real-valued function on $\Bbb{P}$. Define $V_{1}, V_{k-1}:(0,\infty)\to\Bbb{R}$ by
$$V_{1}(x)=\sum_{p\leq x}u(p)~~\textrm{and}~~V_{k-1}(x)=\sum_{p_{1}\cdots{p_{k-1}}\leq x}u(p_{1})\cdots u(p_{k-1}),$$ we have
\begin{align*}
  \sum_{p_{1}\cdots{p_{k}}\leq x}u(p_{1})\cdots u(p_{k})
  =&\sum_{p_{k}\leq{\sqrt{x}}}u(p_{k})V_{k-1}\Big(\frac{x}{p_{k}}\Big)\\
   &+\sum_{p_{1}\cdots{p_{k-1}}\leq{\sqrt{x}}}u(p_{1})\cdots u(p_{k-1})V_{1}\Big(\frac{x}{p_{1}\cdots{p_{k-1}}}\Big)\\
   &-V_{1}(\sqrt{x})V_{k-1}(\sqrt{x}).
\end{align*}
\end{corollary}
\begin{proof}
Set $f=u\chi_{\Bbb{P}}$, $f(n)=\left\{
  \begin{array}{ll}
  u(n) &n\in\Bbb{P}\\
  0 &n\not\in\Bbb{P}
  \end{array}
  \right.$ in Corollary \ref{cor}.
\end{proof}

\subsection{Polylogarithmic functions}
 Polylogarithmic functions $Li_{n}:[0,1]\to\Bbb{R}$ are defined by (see \cite{Lewin})
\begin{align*}
  &Li_{2}(x)=-\int_{0}^{x}\frac{\log(1-t)}{t}dt=\sum_{k=1}^{\infty}\frac{x^{k}}{k^{2}},\\
  &Li_{n}(x)=\int_{0}^{x}\frac{Li_{n-1}(t)}{t}dt=\sum_{k=1}^{\infty}\frac{x^{k}}{k^{n}}~~(n\geq 3).
\end{align*}
They will appear in a natural way, as in \cite{Popa3}, in the proof of multiple Mertens evaluations.

We will need the following integral later.
\begin{proposition}
Denote by $a=\log2$. For any natural number $m$, we have
\begin{equation}\label{eq.int}
  \begin{split}
    \int_{0+0}^{\frac{1}{2}}\frac{\log^{m}(1-x)}{x}dx
  &=(-a)^{m+1}+(-1)^{m}m!\zeta(m+1)\\
  &\quad+(-1)^{m-1}\sum_{s=1}^{m}A_{m}^{s}a^{m-s}Li_{s+1}\Big(\frac{1}{2}\Big).
  \end{split}
\end{equation}
\end{proposition}
\begin{proof}
From the formula of integration by parts and $\log(1-x)=-\sum_{k=1}^{\infty}\frac{x^{k}}{k}$, we obtain
\begin{align*}
  \int_{0+0}^{\frac{1}{2}}\frac{\log^{m}(1-x)}{x}dx
  &=\log x\log^{m}(1-x)\Big|_{0+0}^{\frac{1}{2}}+m\int_{0+0}^{\frac{1}{2}}\log x\frac{\log^{m-1}(1-x)}{1-x}dx\\
  &=(-a)^{m+1}+m\int_{\frac{1}{2}}^{1-0}\frac{\log^{m-1}x\log(1-x)}{x}dx\\
  &=(-a)^{m+1}+m\int_{\frac{1}{2}}^{1-0}\frac{\log^{m-1}x}{x}\left(-\sum_{k=1}^{\infty}\frac{x^{k}}{k}\right)dx\\
  &=(-a)^{m+1}-m\sum_{k=1}^{\infty}\frac{1}{k}\int_{\frac{1}{2}}^{1-0}x^{k-1}\log^{m-1}xdx.
\end{align*}
Using the following indefinite integral, which can also be derived from the formula of integration by parts,
\begin{align*}
  \int x^{k-1}\log^{m-1}xdx=\sum_{s=1}^{m}(-1)^{s-1}A_{m-1}^{s-1}\frac{x^{k}\log^{m-s}x}{k^{s}}+\rm{constant},
\end{align*}
we have
\begin{align*}
  \int_{\frac{1}{2}}^{1-0}x^{k-1}\log^{m-1}xdx=(-1)^{m-1}(m-1)!\frac{1}{k^{m}}
   -\sum_{s=1}^{m}(-1)^{s-1}A_{m-1}^{s-1}\frac{(-a)^{m-s}}{k^{s}2^{k}}.
\end{align*}
By the definition of polylogarithmic functions $Li_{n}~(n\geq 2)$, we get
\begin{align*}
  &\int_{0+0}^{\frac{1}{2}}\frac{\log^{m}(1-x)}{x}dx\\
  &=(-a)^{m+1}-m\sum_{k=1}^{\infty}\frac{1}{k}\left((-1)^{m-1}(m-1)!\frac{1}{k^{m}}
   -\sum_{s=1}^{m}(-1)^{s-1}A_{m-1}^{s-1}\frac{(-a)^{m-s}}{k^{s}2^{k}}\right)\\
  &=(-a)^{m+1}+(-1)^{m}m!\sum_{k=1}^{\infty}\frac{1}{k^{m+1}}
   +(-1)^{m-1}\sum_{s=1}^{m}A_{m}^{s}a^{m-s}\sum_{k=1}^{\infty}\frac{1}{k^{s+1}2^{k}}\\
  &=(-a)^{m+1}+(-1)^{m}m!\zeta(m+1)
   +(-1)^{m-1}\sum_{s=1}^{m}A_{m}^{s}a^{m-s}Li_{s+1}\left(\frac{1}{2}\right),
\end{align*}
which completes the proof.
\end{proof}

\section{Proofs of the main results}
In this section, we will prove our main results (Theorems \ref{first} and \ref{thm} above).

For $k=1,~2$, and $3$, it is easy to verify the result (compare with \cite{Popa2,Popa3}).
We proceed by induction on $k$. Assume that the result holds for any positive integer $<k$, we desire to show that it holds for $k$.

For simplification of notations, we set
\begin{equation}\label{simple}
   \sum_{p_{1}\cdots p_{s}\leq x}\frac{1}{p_{1}\cdots p_{s}}=P_{s}(\log_{2}x)+R_{s}(x),
\end{equation}
where, with $B$ is the Mertens constant as before,
\begin{equation}\label{simplified}
   P_{s}(y)=(y+B)^{s}+\sum_{m=2}^{s}C_{s}^{m}a_{m}(y+B)^{s-m},~~R_{s}(x)=O\left(\frac{(\log_{2}x)^{s-1}}{\log x}\right).
\end{equation}
For example, $P_{1}(y)=y+B$, $P_{2}(y)=(y+B)^{2}-\zeta(2)$, $P_{3}(y)=(y+B)^{3}-3\zeta(2)(y+B)+2\zeta(3)$, they have been appeared in Popa's papers \cite{Popa2} and \cite{Popa3}.

It is clear that
$$P_{s}(y)=P_{1}(y)^{s}+\sum_{m=2}^{s}C_{s}^{m}a_{m}P_{1}(y)^{s-m},$$
and with a natural notation $P_{0}(y)=1$,
$$P_{s}(y-a)=\sum_{t=0}^{s}C_{s}^{t}(-1)^{t}a^{t}P_{s-t}(y).$$

It should be noted that our proof is always under the inductive hypothesis, and we will not repeat it.

First, we prove the following evaluation.
\begin{proposition}\label{main}
Suppose that $p_{1},\cdots,{p_{k}}$ are primes. Then we have
\begin{equation}\label{proposition 3.1}
\begin{aligned}
        \sum_{p_{1}\cdots{p_{k}}\leq x}\frac{1}{p_{1}\cdots p_{k}}
       =&\sum_{p_{1}\leq{\sqrt{x}}}\frac{1}{p_{1}}P_{k-1}\left(\log_{2}\Big(\frac{x}{p_{1}}\Big)\right)\\
        &+\sum_{p_{1}\cdots{p_{k-1}}\leq{\sqrt{x}}}\frac{1}{p_{1}\cdots p_{k-1}}P_{1}\left(\log_{2}\Big(\frac{x}{p_{1}\cdots{p_{k-1}}}\Big)\right)\\
        &-P_{1}(\log_{2}\sqrt{x})P_{k-1}(\log_{2}\sqrt{x})+O\left(\frac{(\log_{2}x)^{k-1}}{\log x}\right)\\
       =&A+B-C+O\left(\frac{(\log_{2}x)^{k-1}}{\log x}\right),
\end{aligned}
\end{equation}
with the notations
\begin{align*}
          A&=\sum_{p_{1}\leq{\sqrt{x}}}\frac{1}{p_{1}}P_{k-1}\left(\log_{2}\Big(\frac{x}{p_{1}}\Big)\right),\\
          B&=\sum_{p_{1}\cdots{p_{k-1}}\leq{\sqrt{x}}}\frac{1}{p_{1}\cdots p_{k-1}}P_{1}\left(\log_{2}\Big(\frac{x}{p_{1}\cdots{p_{k-1}}}\Big)\right),\\
          C&=P_{1}(\log_{2}\sqrt{x})P_{k-1}(\log_{2}\sqrt{x}).
\end{align*}
\end{proposition}
\begin{proof}
Define $V_{1},V_{k-1}:(0,\infty)\to\Bbb{R}$ by
$$V_{1}(x)=\sum_{p\leq x}\frac{1}{p}~~\textrm{and}~~V_{k-1}(x)=\sum_{p_{1}\cdots p_{k-1}\leq x}\frac{1}{p_{1}\cdots p_{k-1}}.$$
Using Corollary \ref{cor.sep} and the simplified notations (\ref{simplified}), we have
\begin{align*}
  &\sum_{p_{1}\cdots{p_{k}}\leq x}\frac{1}{p_{1}\cdots p_{k}}\\
  &=\sum_{p_{1}\leq{\sqrt{x}}}\frac{1}{p_{1}}V_{k-1}\Big(\frac{x}{p_{1}}\Big)
   +\sum_{p_{1}\cdots{p_{k-1}}\leq{\sqrt{x}}}\frac{1}{p_{1}\cdots p_{k-1}}V_{1}\Big(\frac{x}{p_{1}\cdots{p_{k-1}}}\Big)
   -V_{1}(\sqrt{x})V_{k-1}(\sqrt{x})\\
  &=\sum_{p_{1}\leq{\sqrt{x}}}\frac{1}{p_{1}}\left[P_{k-1}\left(\log_{2}\Big(\frac{x}{p_{1}}\Big)\right)
   +R_{k-1}\Big(\frac{x}{p_{1}}\Big)\right]\\
   &\quad+\sum_{p_{1}\cdots{p_{k-1}}\leq{\sqrt{x}}}\frac{1}{p_{1}\cdots p_{k-1}}\left[P_{1}\left(\log_{2}\Big(\frac{x}{p_{1}\cdots{p_{k-1}}}\Big)\right)+R_{1}\Big(\frac{x}{p_{1}\cdots{p_{k-1}}}\Big)\right]\\
   &\quad-\Big({P_{1}(\log_{2}\sqrt{x})+R_{1}(\sqrt{x})}\Big)\Big({P_{k-1}(\log_{2}\sqrt{x})+R_{k-1}(\sqrt{x})}\Big)\\
  &=\sum_{p_{1}\leq{\sqrt{x}}}\frac{1}{p_{1}}P_{k-1}\left(\log_{2}\Big(\frac{x}{p_{1}}\Big)\right)
   +\sum_{p_{1}\cdots{p_{k-1}}\leq{\sqrt{x}}}\frac{1}{p_{1}\cdots p_{k-1}}P_{1}\left(\log_{2}\Big(\frac{x}{p_{1}\cdots{p_{k-1}}}\Big)\right)\\
   &\quad-P_{1}(\log_{2}\sqrt{x})P_{k-1}(\log_{2}\sqrt{x})+R_{k}(x),
\end{align*}
where
\begin{align*}
  R_{k}(x)=&\sum_{p_{1}\leq{\sqrt{x}}}\frac{1}{p_{1}}R_{k-1}\Big(\frac{x}{p_{1}}\Big)
             +\sum_{p_{1}\cdots{p_{k-1}}\leq{\sqrt{x}}}\frac{1}{p_{1}\cdots p_{k-1}}R_{1}\Big(\frac{x}{p_{1}\cdots{p_{k-1}}}\Big)\\
           &-P_{k-1}(\log_{2}\sqrt{x})R_{1}(\sqrt{x})-P_{1}(\log_{2}\sqrt{x})R_{k-1}\left(\sqrt{x}\right)
            -R_{1}(\sqrt{x})R_{k-1}\left(\sqrt{x}\right).
\end{align*}
Since $p_{i}\geq 2~(1\leq i\leq{k-1})$ and
$$\sum_{p_{i}\leq\sqrt{x}}\frac{1}{p_{i}}\frac{1}{\log({x}/{p_{i}})}
          \leq\sum_{p_{i}\leq\sqrt{x}}\frac{1}{p_{i}}\frac{1}{\log({x}/{\sqrt{x}})}
                          =O\left(\frac{\log_{2}x}{\log x}\right),$$
$$\frac{(\log_{2}x)^{k-2}}{\log^{2}x}=O\left(\frac{(\log_{2}x)^{k-1}}{\log x}\right),$$
we obtain
\begin{align*}
  \sum_{p_{1}\leq{\sqrt{x}}}\frac{1}{p_{1}}R_{k-1}\Big(\frac{x}{p_{1}}\Big)
  &=\sum_{p_{1}\leq{\sqrt{x}}}\frac{1}{p_{1}}O\left(\frac{(\log_{2}({x}/{p_{1}}))^{k-2}}{\log({x}/{p_{1}})}\right)
   =O\left(\sum_{p_{1}\leq{\sqrt{x}}}\frac{1}{p_{1}}\frac{(\log_{2}({x}/{p_{1}}))^{k-2}}{\log({x}/{p_{1}})}\right)\\
  &=O\left(\left(\log_{2}({x}/{2})\right)^{k-2}\sum_{p_{1}\leq{\sqrt{x}}}\frac{1}{p_{1}}\frac{1}{\log({x}/{p_{1}})}\right)
   =O\left(\frac{(\log_{2}x)^{k-1}}{\log x}\right)
\end{align*}
and
\begin{align*}
  \sum_{p_{1}\cdots{p_{k-1}}\leq{\sqrt{x}}}&\frac{1}{p_{1}\cdots
        p_{k-1}}R_{1}\Big(\frac{x}{p_{1}\cdots{p_{k-1}}}\Big)
   =\sum_{p_{1}\cdots{p_{k-1}}\leq{\sqrt{x}}}\frac{1}{p_{1}\cdots p_{k-1}}O\left(\frac{1}{\log(\frac{x}{p_{1}\cdots{p_{k-1}}})}\right)\\
  &=O\left(\sum_{p_{1}\cdots{p_{k-1}}\leq{\sqrt{x}}}\frac{1}{p_{1}\cdots p_{k-1}}\frac{1}{\log(\frac{x}{p_{1}\cdots{p_{k-1}}})}\right)
   =O\left(\frac{(\log_{2}x)^{k-1}}{\log x}\right).
\end{align*}
Notice that
\begin{align*}
  &\quad\quad P_{k-1}(\log_{2}\sqrt{x})R_{1}(\sqrt{x})
    =P_{k-1}(\log_{2}\sqrt{x})O\left(\frac{1}{\log x}\right)
    =O\left(\frac{(\log_{2})^{k-1}}{\log x}\right),\\
  &\quad\quad P_{1}(\log_{2}\sqrt{x})R_{k-1}(\sqrt{x})
    =P_{1}(\log_{2}\sqrt{x})O\left(\frac{(\log_{2}x)^{k-2}}{\log x}\right)
    =O\left(\frac{(\log_{2}x)^{k-1}}{\log x}\right),\\
  &\quad\quad R_{1}(\sqrt{x})R_{k-1}(\sqrt{x})
    =O\left(\frac{1}{\log x}\right)O\left(\frac{(\log_{2}x)^{k-2}}{\log x}\right)
    =O\left(\frac{(\log_{2}x)^{k-1}}{\log x}\right),
\end{align*}
we have
\begin{equation*}
   R_{k}(x)=O\left(\frac{(\log_{2}x)^{k-1}}{\log x}\right),
\end{equation*}
which completes the proof.
\end{proof}

To evaluate the terms on the right hand side of (\ref{proposition 3.1}), we need the following propositions.

\begin{proposition}\label{1-lnp/lnx}
The following evaluation holds
\begin{equation*}
  \sum_{p\leq\sqrt{x}}\frac{1}{p}\left[\log\Big(1-\frac{\log p}{\log x}\Big)\right]^{m}
  =\int_{0+0}^{\frac{1}{2}}\frac{\log^{m}(1-x)}{x}dx+O\left(\frac{1}{\log x}\right).
\end{equation*}
\end{proposition}

\begin{proof}
Using the Generalized Newton binomial formula, we have
\begin{equation*}
  \begin{split}
     \left(x+\frac{x^{2}}{2}+\frac{x^{3}}{3}+\cdots\right)^{m}
     =\sum_{k=1}^{\infty}\sum_{\substack{k_{1}+\cdots+k_{m}=k\\k_{1},\cdots,k_{m}\geq{1}}}\dbinom{k}{k_{1}~\cdots~k_{m}}\frac{x^{k_{1}}}{k_{1}}\cdots\frac{x^{k_{m}}}{k_{m}}
     =\sum_{k=1}^{\infty}b_{k}x^{k}
  \end{split}
\end{equation*}
where $b_{1}=\cdots=b_{m-1}=0$, and
$$b_{k}=\sum_{\substack{k_{1}+\cdots+k_{m}=k\\k_{1},\cdots,k_{m}\geq{1}}}\dbinom{k}{k_{1}~\cdots~k_{m}}\frac{1}{k_{1}\cdots{k_{m}}}.$$
Thus for any $|x|<{1}$, we have $log^{m}(1-x)=(-1)^{m}\sum_{k=1}^{\infty}b_{k}x^{k}$ and from the evaluation (see \cite[Proposition 1]{Popa2})
$$\sum_{p\leq\sqrt{x}}\frac{1}{p}\left(\frac{\log p}{\log x}\right)^{k}
=\frac{1}{k\cdot2^{k}}+\frac{1}{2^{k-1}}O\left(\frac{1}{\log x}\right),$$
we obtain
\begin{equation}\label{mid}
\begin{split}
  \sum_{p\leq\sqrt{x}}\frac{1}{p}\left[\log\left(1-\frac{\log p}{\log x}\right)\right]^{m}
    &=(-1)^{m}\sum_{k=1}^{\infty}b_{k}\sum_{p\leq\sqrt{x}}
       \frac{1}{p}\left(\frac{\log p}{\log x}\right)^{k}\\
    &=(-1)^{m}\sum_{k=1}^{\infty}b_{k}\left(\frac{1}{k\cdot2^{k}}+\frac{1}{2^{k-1}}
       O\Big(\frac{1}{\log x}\Big)\right)\\
    &=(-1)^{m}\left(\sum_{k=1}^{\infty}\frac{b_{k}}{k\cdot2^{k}}
       +\left(\sum_{k=1}^{\infty}\frac{b_{k}}{2^{k-1}}\right)
       O\left(\frac{1}{\log x}\right)\right).
\end{split}
\end{equation}
Since
\begin{align*}
  \int_{0+0}^{\frac{1}{2}}\left(\frac{\log^{m}(1-x)}{x}\right)'dx
    &=(-1)^{m}\int_{0+0}^{\frac{1}{2}}\sum_{k=2}^{\infty}b_{k}(k-1)x^{k-2}dx\\
    &=(-1)^{m}\sum_{k=2}^{\infty}b_{k}(k-1)\int_{0+0}^{\frac{1}{2}}x^{k-2}dx\\
    &=(-1)^{m}\sum_{k=2}^{\infty}\frac{b_{k}}{2^{k-1}},
\end{align*}
 we see that the series $\sum_{k=1}^{\infty}\frac{b_{k}}{2^{k-1}}$ is convergent.
 Finally substituting the following integral into (\ref{mid})
 \begin{align*}
   \int_{0+0}^{\frac{1}{2}}\frac{\log^{m}(1-x)}{x}dx
    &=(-1)^{m}\int_{0+0}^{\frac{1}{2}}\sum_{k=1}^{\infty}b_{k}x^{k-1}dx
    =(-1)^{m}\sum_{k=1}^{\infty}b_{k}\int_{0+0}^{\frac{1}{2}}x^{k-1}dx\\
    &=(-1)^{m}\sum_{k=1}^{\infty}b_{k}\frac{1}{k\cdot2^{k}},
 \end{align*}
we obtain
\begin{equation*}
  \sum_{p\leq\sqrt{x}}\frac{1}{p}\left[\log\Big(1-\frac{\log p}{\log x}\Big)\right]^{m}
  =\int_{0+0}^{\frac{1}{2}}\frac{\log^{m}(1-x)}{x}dx+O\left(\frac{1}{\log x}\right),
\end{equation*}
which is the desired result.
\end{proof}

\begin{proposition}\label{lnx/p}
The following evaluation holds
\begin{equation*}
\begin{split}
  \sum_{p\leq\sqrt{x}}\frac{1}{p}\left[\log_{2}\Big(\frac{x}{p}\Big)\right]^{m}
  =&(\log_{2}x)^{m}P_{1}(\log_{2}\sqrt{x})
   +\sum_{t=1}^{m}C^{t}_{m}(\log_{2}x)^{m-t}
   \int_{0+0}^{\frac{1}{2}}\frac{\log^{t}(1-x)}{x}dx\\
  &+O\left(\frac{(\log_{2}x)^{m}}{\log x}\right).
\end{split}
\end{equation*}
\end{proposition}

\begin{proof}
Using Newton binomial formula, we obtain
\begin{align*}
  \sum_{p\leq\sqrt{x}}\frac{1}{p}\left[\log_{2}\Big(\frac{x}{p}\Big)\right]^{m}
    =&\sum_{p\leq\sqrt{x}}\frac{1}{p}\Big(\log(\log x-\log p)-\log_{2}x+\log_{2}x\Big)^{m}\\
    =&\sum_{p\leq\sqrt{x}}\frac{1}{p}\left(\log_{2}x+\log\Big(1-\frac{\log p}{\log x}\Big)\right)^{m}\\
    =&\sum_{p\leq\sqrt{x}}\frac{1}{p}\sum_{t=0}^{m}C_{m}^{t}(\log_{2}x)^{m-t}
     \left(\log\Big(1-\frac{\log p}{\log x}\Big)\right)^{t}\\
    =&(\log_{2}x)^{m}\sum_{p\leq\sqrt{x}}\frac{1}{p}+\sum_{t=1}^{m}C_{m}^{t}(\log_{2}x)^{m-t}
      \sum_{p\leq\sqrt{x}}\frac{1}{p}\left(\log\Big(1-\frac{\log p}{\log x}\Big)\right)^{t}.
\end{align*}
By Proposition \ref{1-lnp/lnx}, we have
\begin{align*}
 \sum_{p\leq\sqrt{x}}\frac{1}{p}\left[\log_{2}\Big(\frac{x}{p}\Big)\right]^{m}
    =&(\log_{2}x)^{m}\left(P_{1}(\log_{2}\sqrt{x})+O\Big(\frac{1}{\log x}\Big)\right)\\
     &+\sum_{t=1}^{m}C_{m}^{t}(\log_{2}x)^{m-t}\left(\int_{0+0}^{\frac{1}{2}}
     \frac{\log^{t}(1-x)}{x}dx+O\Big(\frac{1}{\log x}\Big)\right)\\
    =&(\log_{2}x)^{m}P_{1}(\log_{2}\sqrt{x})+\sum_{t=1}^{m}C_{m}^{t}(\log_{2}x)^{m-t}
     \int_{0+0}^{\frac{1}{2}}\frac{\log^{t}(1-x)}{x}dx\\
      &+O\left(\frac{(\log_{2}x)^{m}}{\log x}\right),
\end{align*}
which completes our proof.
\end{proof}
Using Proposition \ref{1-lnp/lnx} and Proposition \ref{lnx/p}, we can evaluate the sum $A$ on the right hand side of (\ref{proposition 3.1}).

\begin{proposition}\label{prop.A}
The following evaluation holds
\begin{equation*}
  \begin{split}
    A=&\sum_{p_{1}\leq{\sqrt{x}}}\frac{1}{p_{1}}P_{k-1}
      \left(\log_{2}\Big(\frac{x}{p_{1}}\Big)\right)\\
     =&P_{k-1}(\log_{2}x)P_{1}(\log_{2}\sqrt{x})
      +\sum_{c=1}^{k-1}C_{k-1}^{c}P_{k-1-c}(\log_{2}x)
       \int_{0+0}^{\frac{1}{2}}\frac{\log^{c}(1-x)}{x}dx\\
      &+O\left(\frac{(\log_{2}x)^{k-1}}{\log x}\right).
  \end{split}
\end{equation*}
\end{proposition}
\begin{proof}
Using the simplified notations (\ref{simplified}) and Newton binomial formula, we obtain
\begin{align*}
  A=&\sum_{p_{1}\leq{\sqrt{x}}}\frac{1}{p_{1}}\left[\left(\log_{2}\Big(\frac{x}{p_{1}}\Big)+B\right)^{k-1}
     +\sum_{m=2}^{k-1}C_{k-1}^{m}a_{m}\left(\log_{2}\Big(\frac{x}{p_{1}}\Big)+B\right)^{k-1-m}\right]\\
     =&\sum_{p_{1}\leq{\sqrt{x}}}\frac{1}{p_{1}}\left[\sum_{t=0}^{k-1}C_{k-1}^{t}B^{t}\left(\log_{2}\Big(\frac{x}{p_{1}}\Big)\right)^{k-1-t}\right.\\
      &\quad\quad\left.+\sum_{m=2}^{k-1}C_{k-1}^{m}a_{m}\sum_{s=0}^{k-1-m}C_{k-1-m}^{s}B^{s}\left(\log_{2}\Big(\frac{x}{p_{1}}\Big)\right)^{k-1-m-s}\right]\\
     =&\sum_{t=0}^{k-1}C_{k-1}^{t}B^{t}\sum_{p_{1}\leq{\sqrt{x}}}\frac{1}{p_{1}}\left(\log_{2}\Big(\frac{x}{p_{1}}\Big)\right)^{k-1-t}\\
      &+\sum_{m=2}^{k-1}C_{k-1}^{m}a_{m}\sum_{s=0}^{k-1-m}C_{k-1-m}^{s}B^{s}\sum_{p_{1}\leq{\sqrt{x}}}\frac{1}{p_{1}}
       \left(\log_{2}\Big(\frac{x}{p_{1}}\Big)\right)^{k-1-m-s},
\end{align*}
Applying Proposition \ref{lnx/p} we have
\begin{align*}
   A=&\sum_{t=0}^{k-1}C_{k-1}^{t}B^{t}\Bigg[(\log_{2}x)^{k-1-t}P_{1}(\log_{2}\sqrt{x})\\
        &\quad\quad+\sum_{c=1}^{k-1-t}C_{k-1-t}^{c}(\log_{2}x)^{k-1-t-c}
         \int_{0+0}^{\frac{1}{2}}\frac{\log^{c}(1-x)}{x}dx
          +O\left(\frac{(\log_{2}x)^{k-1-t}}{\log x}\right)\Bigg]\\
        &+\sum_{m=2}^{k-1}C_{k-1}^{m}a_{m}\sum_{s=0}^{k-1-m}C_{k-1-m}^{s}B^{s}
         \Bigg[(\log_{2}x)^{k-1-m-s}P_{1}(\log_{2}\sqrt{x})\\
      &\quad\quad+\sum_{c=1}^{k-1-m-s}C_{k-1-m-s}^{c}(\log_{2}x)^{k-1-m-s-c}
       \int_{0+0}^{\frac{1}{2}}\frac{\log^{c}(1-x)}{x}dx\\
        &\quad\quad+O\left(\frac{(\log_{2}x)^{k-1-m-s}}{\log x}\right)\Bigg]\\
     =&\sum_{t=0}^{k-1}C_{k-1}^{t}B^{t}(\log_{2}x)^{k-1-t}P_{1}(\log_{2}\sqrt{x})\\
      &+\sum_{m=2}^{k-1}C_{k-1}^{m}a_{m}\sum_{s=0}^{k-1-m}C_{k-1-m}^{s}B^{s}(\log_{2}x)^{k-1-m-s}P_{1}(\log_{2}\sqrt{x})\\
      &+\sum_{t=0}^{k-1}C_{k-1}^{t}B^{t}\sum_{c=1}^{k-1-t}C_{k-1-t}^{c}(\log_{2}x)^{k-1-t-c}\int_{0+0}^{\frac{1}{2}}\frac{\log^{c}(1-x)}{x}dx\\
      &+\sum_{m=2}^{k-1}C_{k-1}^{m}a_{m}\sum_{s=0}^{k-1-m}C_{k-1-m}^{s}B^{s}
        \sum_{c=1}^{k-1-m-s}C_{k-1-m-s}^{c}(\log_{2}x)^{k-1-m-s-c}\int_{0+0}^{\frac{1}{2}}\frac{\log^{c}(1-x)}{x}dx\\
      &+O\left(\frac{(\log_{2}x)^{k-1}}{\log x}\right)\\
    =&\left((\log_{2}x+B)^{k-1}+\sum_{m=2}^{k-1}C_{k-1}^{m}a_{m}(\log_{2}x+B)^{k-1-m}\right)P_{1}(\log_{2}\sqrt{x})\\
      &+\sum_{c=1}^{k-1}C_{k-1}^{c}\left(\sum_{t=0}^{k-1-c}C_{k-1-c}^{t}B^{t}(\log_{2}x)^{k-1-c-t}\right)\int_{0+0}^{\frac{1}{2}}\frac{\log^{c}(1-x)}{x}dx\\
      &+\sum_{c=1}^{k-1}C_{k-1}^{c}\sum_{m=2}^{k-1-c}C_{k-1-c}^{m}a_{m}
         \left(\sum_{s=0}^{k-1-c-m}C_{k-1-c-m}^{s}B^{s}(\log_{2}x)^{k-1-c-m-s}\right)\int_{0+0}^{\frac{1}{2}}\frac{\log^{c}(1-x)}{x}dx\\
        &+O\left(\frac{(\log_{2}x)^{k-1}}{\log x}\right)\\
    =&\left((\log_{2}x+B)^{k-1}+\sum_{m=2}^{k-1}C_{k-1}^{m}a_{m}(\log_{2}x+B)^{k-1-m}\right)P_{1}(\log_{2}\sqrt{x})\\
        &+\sum_{c=1}^{k-1}C_{k-1}^{c}\left((\log_{2}x+B)^{k-1-c}+\sum_{m=2}^{k-1-c}C_{k-1-c}^{m}a_{m}(\log_{2}x+B)^{k-1-c-m}\right)
          \int_{0+0}^{\frac{1}{2}}\frac{\log^{c}(1-x)}{x}dx\\
        &+O\left(\frac{(\log_{2}x)^{k-1}}{\log x}\right).
\end{align*}
The result is now seen from the simplified notations (\ref{simplified}).
\end{proof}

Now we at the position to prove Theorem \ref{lem.B}.

\noindent\textbf{Proof of Theorem \ref{lem.B}}
The method used here is analogue to that in \cite[Theorem 3.3]{B\u{a}nescu}. If $s=1$, then first we claim that
\begin{align*}
   \sum_{p_{1}\cdots{p_{k}}\leq{x}}&\frac{\log(p_{1}\cdots p_{k})}{p_{1}\cdots p_{k}}=f(x)+A(x),
\end{align*}
where
\begin{align*}
   f(x)=\sum_{l=0}^{k-1}(-1)^{l}A_{k}^{l+1}P_{k-1-l}(\log_{2}x)\cdot \log x
\end{align*}
and \begin{equation}\label{A(x)}A(x)=O\left((\log_{2}x)^{k}\right).\end{equation}

In fact, by the Abel summation formula (see \cite[p. 5]{B\u{a}nescu}), we have
\begin{align*}
  &\sum_{p_{1}\cdots{p_{k}}\leq{x}}\frac{\log(p_{1}\cdots p_{k})}{p_{1}\cdots p_{k}}\\
    &=\left(\sum_{p_{1}\cdots{p_{k}}\leq{x}}\frac{1}{p_{1}\cdots p_{k}}\right)\log x
       -\int_{2}^{x}\left(\sum_{p_{1}\cdots{p_{k}}\leq{t}}\frac{1}{p_{1}\cdots p_{k}}\right)(\log t)^{'}dt\\
    &=\left(P_{k}(\log_{2}x)+R_{k}(x)\right)\log x
    -\int_{2}^{x}\left(P_{k}(\log_{2}t)+R_{k}(t)\right)(\log t)^{'}dt\\
    &=P_{k}(\log_{2}x)\log x+O\left((\log_{2}x)^{k-1}\right)-\left(\log t\cdot
       P_{k}(\log_{2}t)\Big|_{2}^{x}-\int_{2}^{x}\frac{P_{k}^{'}(\log_{2}t)}{t}dt\right)\\
       &\quad-\int_{2}^{x}R_{k}(t)(\log t)^{'}dt\\
    &=\int_{2}^{x}\frac{P_{k}^{'}(\log_{2}t)}{t}dt+\log 2\cdot P_{k}(\log_{2}2)-\int_{2}^{x}R_{k}(t)(\log t)^{'}dt+O\left((\log_{2}x)^{k-1}\right)\\
    &=\int_{\log2}^{\log x}P_{k}^{'}(\log t)dt+O\left((\log_{2}x)^{k}\right),
\end{align*}
since the inner term
\begin{align*}
  \int_{2}^{x}R_{k}(t)(\log t)^{'}dt
  &=\int_{2}^{x}O\left(\frac{(\log_{2}t)^{k-1}}{\log t}\right)(\log t)^{'}dt
   =O\left(\int_{2}^{x}\frac{(\log_{2}t)^{k-1}}{t\log t}dt\right)\\
  &=O\left(\int_{\log2}^{\log x}\frac{\log^{k-1}t}{t}dt\right)
   =O\left(\int_{\log2}^{\log x}\log^{k-1}td(\log t)\right)\\
  &=O\left(\frac{\log^{k}t}{k}\Big|_{\log2}^{\log x}\right)
   =O\left((\log_{2}x)^{k}\right).
\end{align*}
Note that
\begin{align*}
  P_{k}^{'}(y)&=\left((y+B)^{k}+\sum_{m=2}^{k}C_{k}^{m}a_{m}(y+B)^{k-m}\right)^{'}\\
            &=k(y+B)^{k-1}+\sum_{m=2}^{k-1}C_{k}^{m}a_{m}(k-m)(y+B)^{k-1-m}\\
            &=k\left((y+B)^{k-1}+\sum_{m=2}^{k}C_{k-1}^{m}a_{m}(y+B)^{k-1-m}\right).
\end{align*}
Hence
\begin{align*}
  \int_{\log2}^{\log x}P_{k}^{'}(\log t)dt
  &=k\int_{\log2}^{\log x}\left((\log t+B)^{k-1}
    +\sum_{m=2}^{k}C_{k-1}^{m}a_{m}(\log t+B)^{k-1-m}\right)dt
 \end{align*}
 \begin{equation}\label{p'}
   \begin{split}
     &\qquad\qquad=k\int_{\log2}^{\log x}\left(\sum_{s_{1}=0}^{k-1}C_{k-1}^{s_{1}}B^{s_{1}}(\log t)^{k-1-s_{1}}\right.\\
     &\qquad\qquad\quad\quad\left.+\sum_{m=2}^{k}C_{k-1}^{m}a_{m}\sum_{s_{2}=0}^{k-1-m}C_{k-1-m}^{s_{2}}B^{s_{2}}(\log t)^{k-1-m-s_{2}}\right)dt\\
  &\qquad\qquad=k\left(\sum_{s_{1}=0}^{k-1}C_{k-1}^{s_{1}}B^{s_{1}}\int_{\log2}^{\log x}(\log t)^{k-1-s_{1}}dt\right.\\
     &\qquad\qquad\quad\quad\left.+\sum_{m=2}^{k}C_{k-1}^{m}a_{m}
     \sum_{s_{2}=0}^{k-1-m}C_{k-1-m}^{s_{2}}B^{s_{2}}\int_{\log2}^{\log x}(\log t)^{k-1-m-s_{2}}dt\right).
   \end{split}
 \end{equation}
By the indefinite integral $$\int \log^{a}tdt=\sum_{l=0}^{a}(-1)^{l}A_{a}^{l}t(\log t)^{a-l}+\textrm{constant},$$ the first term on the right hand side of (\ref{p'}) becomes to
\begin{align*}
  &\sum_{s_{1}=0}^{k-1}C_{k-1}^{s_{1}}B^{s_{1}}\left(\sum_{l_{1}=0}^{k-1-s_{1}}(-1)^{l_{1}}A_{k-1-s_{1}}^{l_{1}}t(\log t)^{k-1-s_{1}-l_{1}}\right)\Bigg|_{\log2}^{\log x}\\
  =&\sum_{l_{1}=0}^{k-1}(-1)^{l_{1}}A_{k-1}^{l_{1}}\left(\sum_{s_{1}=0}^{k-1-l_{1}}C_{k-1-l_{1}}^{s_{1}}B^{s_{1}}(\log t)^{k-1-s_{1}-l_{1}}\right)\cdot t \Bigg|_{\log2}^{\log x}\\
  =&\sum_{l_{1}=0}^{k-1}(-1)^{l_{1}}A_{k-1}^{l_{1}}(\log t+B)^{k-1-l_{1}}\cdot t\Bigg|_{\log2}^{\log x}
\end{align*}
and the second term becomes to
\begin{align*}
  &\sum_{m=2}^{k}C_{k-1}^{m}a_{m}\sum_{s_{2}=0}^{k-1-m}C_{k-1-m}^{s_{2}}B^{s_{2}}\left(\sum_{l_{2}=0}^{k-1-m-s_{2}}(-1)^{l_{2}}
        A_{k-1-m-s_{2}}^{l_{2}}t(\log t)^{k-1-m-s_{2}-l_{2}}\right)\Bigg|_{\log2}^{\log x}\quad\quad\quad\\
  =&\sum_{l_{2}=0}^{k-1}(-1)^{l_{2}}A_{k-2}^{l_{2}}\sum_{m=2}^{k-1-l_{2}}C_{k-1-l_{2}}^{m}a_{m}\left(
       \sum_{s_{2}=0}^{k-1-l_{2}-m}C_{k-1-l_{2}-m}^{s_{2}}B^{s_{2}}(\log t)^{k-1-m-s_{2}-l_{2}}\right)\cdot t\Bigg|_{\log2}^{\log x}\\
  =&\sum_{l_{2}=0}^{k-1}(-1)^{l_{2}}A_{k-1}^{l_{2}}\sum_{m=2}^{k-1-l_{2}}C_{k-1-l_{2}}^{m}a_{m}(\log t+B)^{k-1-m-l_{2}}\cdot t\Bigg|_{\log2}^{\log x}.
\end{align*}
Hence
\begin{align*}
  &\int_{\log2}^{\log x}P_{k}^{'}(\log t)dt\\
  &=k\sum_{l=0}^{k-1}(-1)^{l}A_{k-1}^{l}\left((\log t+B)^{k-1-l}+\sum_{m=2}^{k-1-l}C_{k-1-l}^{m}a_{m}(\log t+B)^{k-1-l-m}\right)\cdot t\Bigg|_{\log2}^{\log x}\\
  &=k\sum_{l=0}^{k-1}(-1)^{l}A_{k-1}^{l}P_{k-1-l}(\log t)\cdot t\Bigg|_{\log2}^{\log x}\\
  &=k\sum_{l=0}^{k-1}(-1)^{l}A_{k-1}^{l}P_{k-1-l}(\log_{2}x)\cdot\log x
     -k\sum_{l=0}^{k-1}(-1)^{l}A_{k-1}^{l}P_{k-1-l}(\log_{2}2)\cdot \log2,
\end{align*}
and the claim follows.

If $s\geq{2}$, then by the Abel summation formula (see \cite{B\u{a}nescu}), we have
\begin{equation}\label{s2}
\begin{aligned}
   &\sum_{p_{1}\cdots{p_{k}}\leq{x}}\frac{\log^{s}(p_{1}\cdots p_{k})}{p_{1}\cdots p_{k}}\\
   &=\left(\sum_{p_{1}\cdots{p_{k}}\leq{x}}\frac{\log(p_{1}\cdots p_{k})}{p_{1}\cdots p_{k}}\right)\log^{s-1}x
       -\int_{2}^{x}\left(\sum_{p_{1}\cdots{p_{k}}\leq{t}}\frac{\log(p_{1}\cdots p_{k})}{p_{1}\cdots p_{k}}\right)(\log^{s-1}t)^{'}dt\\
   &=(f(x)+A(x))\log^{s-1}x-\int_{2}^{x}(f(t)+A(t))(\log^{s-1}t)^{'}dt\\
   &=f(x)\log^{s-1}x+A(x)\log^{s-1}x-\left(f(t)\log^{s-1}t\Big|_{2}^{x}-\int_{2}^{x}f^{'}(t)\log^{s-1}tdt\right)-\int_{2}^{x}A(t)(\log^{s-1}t)^{'}dt\\
   &=\int_{2}^{x}f^{'}(t)\log^{s-1}tdt+f(2)\log^{s-1}2+N(x),
\end{aligned}
  \end{equation}
where
$$ N(x)=A(x)\log^{s-1}x-(s-1)\int_{2}^{x}A(t)\frac{\log^{s-2}t}{t}dt.$$

Now we evaluate the right hand side of (\ref{s2}) term-wise.

For $N(x)$, by (\ref{A(x)}) we have
\begin{align*}
  N(x)&\ll\left|A(x)\right|\log^{s-1}x+(s-1)\int_{2}^{x}\left|A(t)\right|\frac{\log^{s-2}t}{t}dt\\
      &\ll \log^{s-1}x\cdot \left((\log_{2}x)^{k}\right)+(s-1)\int_{2}^{x}\frac{(\log_{2}t)^{k}\log^{s-2}t}{t}dt
      \end{align*}
and from the integral $$\int(\log t)^{k}t^{s}dt=\sum_{l=0}^{k}(-1)^{l}A_{k}^{l}\frac{t^{s+1}(\log t)^{k-l}}{(s+1)^{l+1}}+\textrm{constant},$$
we have $$\int_{2}^{x}\frac{(\log_{2}t)^{k}\log^{s-2}t}{t}dt\ll \log^{s-1}x\cdot (\log_{2}x)^{k}$$
and \begin{equation}\label{N(x)} N(x)=O\left(\log^{s-1}x\cdot(\log_{2}x)^{k}\right).\end{equation}

For $\int_{2}^{x}f^{'}(t)\log^{s-1}tdt$, first we notice that
\begin{align*}
    f^{'}(x)=&\sum_{l=0}^{k-2}(-1)^{l}A_{k}^{l+1}\left(\frac{(k-1-l)(\log_{2}x+B)^{k-2-l}}{x\log x}\right.\\
             &\quad\quad\quad+\left.\sum_{m=2}^{k-2-l}C_{k-1-l}^{m}a_{m}\frac{(k-1-l-m)(\log_{2}x+B)^{k-2-l-m}}{x\log x}\right)\log x\\
             &+\frac{1}{x}\sum_{l=0}^{k-1}(-1)^{l}A_{k}^{l+1}\Big((\log_{2}x+B)^{k-1-l}+\sum_{m=2}^{k-1-l}C_{k-1-l}^{m}a_{m}(\log_{2}x+B)^{k-1-l-m}\Big)\\
            =&\frac{1}{x}\sum_{l=1}^{k-1}(-1)^{l-1}A_{k}^{l+1}\Big((\log_{2}x+B)^{k-1-l}+\sum_{m=2}^{k-1-l}C_{k-1-l}^{m}a_{m}(\log_{2}x+B)^{k-1-l-m}\Big)\\
             &+\frac{1}{x}\sum_{l=0}^{k-1}(-1)^{l}A_{k}^{l+1}\Big((\log_{2}x+B)^{k-1-l}+\sum_{m=2}^{k-1-l}C_{k-1-l}^{m}a_{m}(\log_{2}x+B)^{k-1-l-m}\Big)\\
            =&k\cdot\frac{1}{x}\Big((\log_{2}x+B)^{k-1}+\sum_{m=2}^{k-1}C_{k-1}^{m}a_{m}(\log_{2}x+B)^{k-1-m}\Big).
\end{align*}
This implies
\begin{align*}
  \int_{2}^{x}f^{'}(t)\log^{s-1}tdt
   =&k\int_{2}^{x}\frac{(\log_{2}t+B)^{k-1}+\sum_{m=2}^{k-1}C_{k-1}^{m}a_{m}(\log_{2}t+B)^{k-1-m}}{t}\log^{s-1}tdt\\
   =&k\int_{\log2}^{\log x}\Big((\log t+B)^{k-1}+\sum_{m=2}^{k-1}C_{k-1}^{m}a_{m}(\log t+B)^{k-1-m}\Big)t^{s-1}dt\\
   =&k\int_{\log2}^{\log x}\Big(\sum_{t=0}^{k-1}C_{k-1}^{t}B^{t}(\log t)^{k-1-t}\Big.\\
    &\quad\quad\Big.+\sum_{m=2}^{k-1}C_{k-1}^{m}a_{m}\sum_{t=0}^{k-1-m}C_{k-1-m}^{t}
     B^{t}(\log t)^{k-1-m-t}\Big)t^{s-1}dt\\
   =&k\left(\sum_{t=0}^{k-1}C_{k-1}^{t}B^{t}\int_{\log2}^{\log x}(\log t)^{k-1-t}t^{s-1}dt\right.\\
    &\quad\quad\left.+\sum_{m=2}^{k-1}C_{k-1}^{m}a_{m}
     \sum_{t=0}^{k-1-m}C_{k-1-m}^{t} B^{t}\int_{\log2}^{\log x}(\log t)^{k-1-m-t}t^{s-1}dt\right).
\end{align*}
If we denote
\begin{align*}
  A_{s}=\sum_{l=0}^{k-1}(-1)^{l+1}\frac{A_{k}^{l+1}}{s^{l+1}}P_{k-1-l}(\log_{2}2)\log^{s}2,
\end{align*}
then
\begin{align*}
    &\int_{2}^{x}f^{'}(t)\log^{s-1}tdt\\
   &=k\left(\sum_{t=0}^{k-1}C_{k-1}^{t}B^{t}\sum_{l=0}^{k-1-t}(-1)^{l}A_{k-1-t}^{l}\frac{t^{s}(\log t)^{k-1-t-l}}{s^{l+1}}\right.\\
     &\quad\quad\left.+\sum_{m=2}^{k-1}C_{k-1}^{m}a_{m}
    \sum_{t=0}^{k-1-m}C_{k-1-m}^{t} B^{t}\sum_{l=0}^{k-1-m-t}(-1)^{l}A_{k-1-m-t}^{l}\frac{t^{s}(\log t)^{k-1-m-t-l}}{s^{l+1}}\right)\Bigg|_{\log2}^{\log x}\\
   &=k\left[\sum_{l=0}^{k-1}(-1)^{l}\frac{A_{k-1}^{l}}{s^{l+1}}\Big(\sum_{t=0}^{k-1-l}C_{k-1-l}^{t}
        B^{t}(\log t)^{k-1-l-t}\Big)t^{s}\right.\\
    &\quad\quad\left.+\sum_{l=0}^{k-1}(-1)^{l}\frac{A_{k-1}^{l}}{s^{l+1}}\sum_{m=2}^{k-1-l}C_{k-1-l}^{m}a_{m}\Big(\sum_{t=0}^{k-1-m-l}C_{k-1-m-l}^{t}
     B^{t}(\log t)^{k-1-m-l-t}\Big)t^{s}\right]\Bigg|_{\log2}^{\log x}\\
   &=\sum_{l=0}^{k-1}(-1)^{l}\frac{A_{k}^{l+1}}{s^{l+1}}P_{k-1-l}(\log_{2}x)\cdot\log^{s}x+A_{s}.
\end{align*}
Therefore
\begin{align*}
  \sum_{p_{1}\cdots{p_{k}}\leq{x}}\frac{\log^{s}(p_{1}\cdots p_{k})}{p_{1}\cdots p_{k}}
    =&\sum_{l=0}^{k-1}(-1)^{l}\frac{A_{k}^{l+1}}{s^{l+1}}P_{k-1-l}(\log_{2}x)\cdot \log^{s}x
     +f(2)\log^{s-1}2\\
     &+ O\left(\log^{s-1}x\cdot (\log_{2}x)^{k}\right),
\end{align*}
the proof of Theorem \ref{lem.B} is completed.

\begin{proposition}\label{prop.B}
The following evaluation holds
\begin{align*}
  B&=\sum_{p_{1}\cdots{p_{k-1}}\leq{\sqrt{x}}}\frac{1}{p_{1}\cdots p_{k-1}}P_{1}\left(\log_{2}\Big(\frac{x}{p_{1}\cdots{p_{k-1}}}\Big)\right)\\
    &=P_{1}(\log_{2}x)P_{k-1}(\log_{2}\sqrt{x})+
      \sum_{l=0}^{k-2}(-1)^{l+1}A_{k-1}^{l+1}P_{k-2-l}(\log_{2}\sqrt{x})Li_{l+2}\left(\frac{1}{2}\right)\\
      &\quad+O\left(\frac{(\log_{2}x)^{k-1}}{\log x}\right).
\end{align*}
\end{proposition}

\begin{proof}
The simplified notation (\ref{simplified}) gives
\begin{align*}
  B&=\sum_{p_{1}\cdots{p_{k-1}}\leq{\sqrt{x}}}\frac{1}{p_{1}\cdots p_{k-1}}P_{1}\left(\log_{2}\Big(\frac{x}{p_{1}\cdots{p_{k-1}}}\Big)\right)\\
     &=\sum_{p_{1}\cdots{p_{k-1}}\leq{\sqrt{x}}}\frac{1}{p_{1}\cdots p_{k-1}}\left(\log_{2}\Big(\frac{x}{p_{1}\cdots{p_{k-1}}}\Big)+B\right)\\
     &=\sum_{p_{1}\cdots{p_{k-1}}\leq{\sqrt{x}}}\frac{1}{p_{1}\cdots p_{k-1}}\Big(\log\big(\log x-\log(p_{1}\cdots p_{k})\big)-\log_{2}x+\log_{2}x+B\Big)\\
     &=(\log_{2}x+B)\sum_{p_{1}\cdots{p_{k-1}}\leq{\sqrt{x}}}\frac{1}{p_{1}\cdots p_{k-1}}
        +\sum_{p_{1}\cdots{p_{k-1}}\leq{\sqrt{x}}}\frac{1}{p_{1}
       \cdots p_{k-1}}\log\left(1-\frac{\log(p_{1}\cdots p_{k-1})}{\log x}\right).
\end{align*}

\noindent Then from the series expansion $\log(1-x)$, we have
\begin{align*}
     B&=P_{1}(\log_{2}x)P_{k-1}(\log_{2}\sqrt{x})+
        \sum_{p_{1}\cdots{p_{k-1}}\leq{\sqrt{x}}}\frac{1}{p_{1}\cdots p_{k-1}}\left(-\sum_{s=1}^{\infty}\frac{\log^{s}(p_{1}\cdots p_{k-1})}{s(\log x)^{s}}\right)\\
        &\quad+O\left(\frac{(\log_{2}x)^{k-1}}{\log x}\right)\\
     &=P_{1}(\log_{2}x)P_{k-1}(\log_{2}\sqrt{x})
        -\sum_{s=1}^{\infty}\frac{1}{s}\left(\frac{1}{(\log x)^{s}}\sum_{p_{1}\cdots{p_{k-1}}\leq{\sqrt{x}}}\frac{\log^{s}(p_{1}\cdots p_{k-1})}{p_{1}\cdots p_{k-1}}\right)\\
        &\quad+O\left(\frac{(\log_{2}x)^{k-1}}{\log x}\right)\\
     &=P_{1}(\log_{2}x)P_{k-1}(\log_{2}\sqrt{x})+M+O\left(\frac{(\log_{2}x)^{k-1}}{\log x}\right),
\end{align*}
where
\begin{align*}
  M&=-\sum_{s=1}^{\infty}\frac{1}{s}\left(\frac{1}{(\log x)^{s}}\sum_{p_{1}\cdots{p_{k-1}}\leq{\sqrt{x}}}\frac{\log^{s}(p_{1}\cdots p_{k-1})}{p_{1}\cdots p_{k-1}}\right)\\
       &= -\sum_{s=1}^{\infty}\frac{1}{s}\frac{1}{2^{s}}\left(\sum_{l=0}^{k-2}(-1)^{l}\frac{A_{k-1}^{l+1}}{s^{l+1}}P_{k-2-l}(\log_{2}\sqrt{x})
        +O\left(\frac{(\log_{2}x)^{k-1}}{\log x}\right)\right),
\end{align*}
the last equation holds by using Corollary \ref{cor2}. Since the series $\sum_{s=1}^{\infty}\frac{1}{2^{s}\cdot s}$ is convergent, we obtain
\begin{align*}
  M&=\sum_{l=0}^{k-2}(-1)^{l+1}\frac{A_{k-1}^{l+1}}{s^{l+1}}P_{k-2-l}(\log_{2}\sqrt{x})
      \sum_{s=1}^{\infty}\frac{1}{s^{l+2}\cdot2^{s}}+O\left(\frac{(\log_{2}x)^{k-1}}{\log x}\right)\\
   &=\sum_{l=0}^{k-2}(-1)^{l+1}\frac{A_{k-1}^{l+1}}{s^{l+1}}P_{k-2-l}(\log_{2}\sqrt{x})Li_{l+2}
      \left(\frac{1}{2}\right)+O\left(\frac{(\log_{2}x)^{k-1}}{\log x}\right)
\end{align*}
and the result is established.
\end{proof}

Finally, we prove Theorem \ref{thm}.

\noindent\textbf{Proof of Theorem \ref{thm}}
By Proposition \ref{main}, we have
\begin{equation}\label{thm1}
\begin{aligned}
  \sum_{p_{1}\cdots{p_{k}}\leq x}\frac{1}{p_{1}\cdots p_{k}}=A+B-C+O\left(\frac{(\log_{2}x)^{k-1}}{\log x}\right).
\end{aligned}
\end{equation}
Denote by $a=\log2$. Then substituting the results of Propositions \ref{prop.A} and \ref{prop.B} into (\ref{thm1}), we have
\begin{align*}
  \sum_{p_{1}\cdots{p_{k}}\leq x}\frac{1}{p_{1}\cdots p_{k}}=P_{k}(\log_{2}x)+O\left(\frac{(\log_{2}x)^{k-1}}{\log x}\right),
\end{align*}
where
\begin{align*}
  P_{k}(y)=&P_{1}(y-a)P_{k-1}(y)+\sum_{c=1}^{k-1}C_{k-1}^{c}P_{k-1-c}(y)\int_{0+0}^{\frac{1}{2}}\frac{log^{c}(1-x)}{x}dx\\
            &+P_{1}(y)P_{k-1}(y-a)+\sum_{l=0}^{k-2}(-1)^{l+1}A_{k-1}^{l+1}P_{k-2-l}(y-a)Li_{l+2}\left(\frac{1}{2}\right)\\
            &-P_{1}(y-a)P_{k-1}(y-a)\\
          =&P_{1}(y-a)P_{k-1}(y)+aP_{k-1}(y-a)
            +\sum_{c=1}^{k-1}C_{k-1}^{c}P_{k-1-c}(y)
            \int_{0+0}^{\frac{1}{2}}\frac{\log^{c}(1-x)}{x}dx\\
            &+\sum_{l=1}^{k-1}(-1)^{l}A_{k-1}^{l}P_{k-1-l}(y-a)Li_{l+1}\left(\frac{1}{2}\right).
\end{align*}

Now we evaluate $P_{k}(y)$ carefully.
Notice that for any $s\leq{k-1}$, \begin{equation*}P_{s}(y-a)=\sum_{t=0}^{s}C_{s}^{t}(-1)^{t}a^{t}P_{s-t}(y)\end{equation*} and from the integral (\ref{eq.int}), we obtain
\begin{align*}
   P_{k}(y)=&P_{1}(y-a)P_{k-1}(y)+a\sum_{t=0}^{k-1}C_{k-1}^{t}(-1)^{t}a^{t}P_{k-1-t}(y)\\
            &+\sum_{t=1}^{k-1}C_{k-1}^{t}P_{k-1-t}(y)
             \left((-a)^{t+1}+(-1)^{t}t!\zeta(t+1)+(-1)^{t-1}\sum_{s=1}^{t}A_{t}^{s}a^{t-s}Li_{s+1}\left(\frac{1}{2}\right)\right)\\
            &+\sum_{t=1}^{k-1}(-1)^{t}A_{k-1}^{t}P_{k-1-t}(y-a)Li_{t+1}\left(\frac{1}{2}\right).
\end{align*}
Consider the identity given by operations on combinatorial numbers,
\begin{align*}
  &\sum_{t=1}^{k-1}(-1)^{t-1}C_{k-1}^{t}P_{k-1-t}(y)\sum_{s=1}^{t}A_{t}^{s}a^{t-s}Li_{s+1}\left(\frac{1}{2}\right)\\
 =&\sum_{s=1}^{k-1}\sum_{t=s}^{k-1}(-1)^{t-1}C_{k-1}^{t}A_{t}^{s}a^{t-s}P_{k-1-t}(y)Li_{s+1}\left(\frac{1}{2}\right)\\
 =&\sum_{s=1}^{k-1}\sum_{t=0}^{k-1-s}(-1)^{t+s-1}C_{k-1}^{t+s}A_{t+s}^{s}a^{t}P_{k-1-t-s}(y)Li_{s+1}\left(\frac{1}{2}\right)\\
 =&\sum_{s=1}^{k-1}(-1)^{s-1}A_{k-1}^{s}\left(\sum_{t=0}^{k-1-s}(-1)^{t}C_{k-1-s}^{t}a^{t}P_{k-1-t-s}(y)\right)Li_{s+1}\left(\frac{1}{2}\right)\\
 =&\sum_{s=1}^{k-1}(-1)^{s-1}A_{k-1}^{s}P_{k-1-s}(y-a)Li_{s+1}\left(\frac{1}{2}\right).
\end{align*}
Then substitution it to the above identity for $P_{k}(y)$, we obtain
\begin{align*}
   P_{k}(y)=&P_{1}(y)P_{k-1}(y)-aP_{k-1}(y)+aP_{k-1}(y)+\sum_{t=1}^{k-1}C_{k-1}^{t}(-1)^{t}a^{t+1}P_{k-1-t}(y)\\
            &+\sum_{t=1}^{k-1}C_{k-1}^{t}(-1)^{t+1}a^{t+1}P_{k-1-t}(y)+\sum_{t=1}^{k-1}C_{k-1}^{t}(-1)^{t}t!\zeta(t+1)P_{k-1-t}(y)\\
            &+\sum_{s=1}^{k-1}(-1)^{s-1}A_{k-1}^{s}P_{k-1-s}(y-a)Li_{s+1}\left(\frac{1}{2}\right)\\
            &+\sum_{t=1}^{k-1}(-1)^{t}A_{k-1}^{t}P_{k-1-t}(y-a)Li_{t+1}\left(\frac{1}{2}\right)\\
           =&P_{1}(y)P_{k-1}(y)+\sum_{t=1}^{k-1}C_{k-1}^{t}(-1)^{t}t!\zeta(t+1)P_{k-1-t}(y).
\end{align*}
Reusing the simplified notations (\ref{simplified}), we have
\begin{align*}
   P_{k}(y)=&P_{1}(y)\left(P_{1}(y)^{k-1}+\sum_{m=2}^{k-1}C_{k-1}^{m}a_{m}P_{1}(y)^{k-1-m}\right)\\
             &+\sum_{t=1}^{k-1}C_{k-1}^{t}(-1)^{t}t!\zeta(t+1)
             \left(P_{1}(y)^{k-1-t}+\sum_{m=2}^{k-1-t}C_{k-1-t}^{m}a_{m}P_{1}(y)^{k-1-t-m}\right)\\
           =&P_{1}(y)^{k}+\sum_{m=2}^{k-1}C_{k-1}^{m}a_{m}P_{1}(y)^{k-m}
             +\sum_{t=2}^{k}C_{k-1}^{t-1}(-1)^{t-1}(t-1)!\zeta(t)P_{1}(y)^{k-t}\\
             &+\sum_{t=2}^{k}C_{k-1}^{t-1}(-1)^{t-1}(t-1)!\zeta(t)\sum_{m=2}^{k-t}C_{k-t}^{m}a_{m}P_{1}(y)^{k-t-m}\\
           =&P_{1}(y)^{k}+\sum_{m=2}^{k-1}C_{k-1}^{m}a_{m}P_{1}(y)^{k-m}
             +\sum_{t=2}^{k-1}C_{k-1}^{t-1}(-1)^{t-1}(t-1)!\zeta(t)P_{1}(y)^{k-t}\\
             &+(-1)^{k-1}(k-1)!\zeta(k)+\sum_{t=2}^{k-2}(-1)^{t-1}C_{k-1}^{t-1}(t-1)!\zeta(t)a_{k-t}\\
             &+\sum_{m=4}^{k-1}\sum_{t=2}^{m-2}(-1)^{t-1}C_{k-1}^{t-1}(t-1)!\zeta(t)C_{k-t}^{m-t}a_{m-t}P_{1}(y)^{k-m}.
 \end{align*}
 Thus from the definition of the sequence $\{a_{n}\}$ (\ref{an}), we get
 \begin{align*}
    P_{k}(y)=&P_{1}(y)^{k}+\sum_{m=2}^{3}\left(C_{k-1}^{m}a_{m}+C_{k-1}^{m-1}(-1)^{m-1}(m-1)!\zeta(m)\right)P_{1}(y)^{k-m}\\
             &+\sum_{m=4}^{k-1}\Bigg(C_{k-1}^{m}a_{m}+C_{k-1}^{m-1}(-1)^{m-1}(m-1)!\zeta(m)\\
              &\quad\quad\quad\left.+\sum_{t=2}^{m-2}(-1)^{t-1}C_{k-1}^{t-1}(t-1)!\zeta(t)C_{k-t}^{m-t}a_{m-t}\right)P_{1}(y)^{k-m}\\
             &+\sum_{t=1}^{k-3}C_{k-1}^{t}(-1)^{t}t!\zeta(t+1)a_{k-1-t}+(-1)^{k-1}(k-1)!\zeta(k).
 \end{align*}
Then by using the equation $$ \sum_{t=2}^{m-2}(-1)^{t-1}C_{k-1}^{t-1}(t-1)!\zeta(t)C_{k-t}^{m-t}a_{m-t}
                         =\sum_{t=1}^{m-3}(-1)^{t}C_{k-1}^{t}t!\zeta(t+1)C_{k-t-1}^{m-t-1}a_{m-t-1},$$
 we obtain
 \begin{align*}
            P_{k}(y)=&P_{1}(y)^{k}+\sum_{m=2}^{3}\left(C_{k-1}^{m}a_{m}+C_{k-1}^{m-1}a_{m}\right)P_{1}(y)^{k-m}\\
              &+\sum_{m=4}^{k-1}\Bigg(C_{k-1}^{m}a_{m}+C_{k-1}^{m-1}(-1)^{m-1}(m-1)!\zeta(m)\\
              &\quad\quad\quad\left.+C_{k-1}^{m-1}\sum_{t=1}^{m-3}(-1)^{t}C_{m-1}^{t}t!\zeta(t+1)a_{m-1-t}\right)P_{1}(y)^{k-m}+a_{k}\\
           =&P_{1}(y)^{k}+\sum_{m=2}^{k-1}C_{k}^{m}a_{m}P_{1}(y)^{k-m}+a_{k},
\end{align*}
since $C_{k-1}^{m}+C_{k-1}^{m-1}=C_{k}^{m}$. This completes the proof of Theorem \ref{thm}.

\section*{Acknowledgement} We would like to thank Professor M. A. Korolev for his interested in this paper and for sending his work to us.

\end{document}